\documentclass[11pt]{extarticle}

\usepackage{comment}
\usepackage[utf8]{inputenc}
\usepackage[all]{xy}
\usepackage{bm}
\usepackage{enumitem}
\usepackage{colonequals,amsmath,amssymb,amsthm,mathtools}
\usepackage{mathrsfs}
\usepackage[ruled,vlined]{algorithm2e}
\usepackage{float}
\usepackage{tikz,caption,subcaption}
\usetikzlibrary{positioning, shapes, plotmarks, decorations.markings, arrows}
\tikzstyle{vertex} = [circle, draw, fill, minimum size=4pt, inner sep=0pt]

\graphicspath{ {./images/} }
\usepackage{caption}
\usepackage{todonotes}
\usepackage[colorlinks=true]{hyperref}
\hypersetup{colorlinks=true,  linkcolor=blue,  filecolor=blue,  citecolor=blue, urlcolor=blue}
\usepackage{geometry}
\usepackage{comment}
\usepackage[capitalize,noabbrev]{cleveref}
\usepackage{pdfpages}
\usepackage{microtype}
\usepackage[justification=centering]{caption}
\usepackage{titlesec}
\usepackage{indentfirst}
\usepackage[style=ieee-alphabetic, maxbibnames=99, backend=biber,giveninits=true,noerroretextools=true,url=false,doi=true,isbn=false]{biblatex}
\newtheorem{theorem}{Theorem}
\newtheorem{proposition}[theorem]{Proposition}
\newtheorem{lemma}[theorem]{Lemma}

\theoremstyle{definition}
\newtheorem{definition}[theorem]{Definition}

\newtheorem{remark}[theorem]{Remark}

\newtheorem{example}[theorem]{Example}

\definecolor{ggreen}{RGB}{0, 94, 25}

\newcommand{\can}{{\operatorname{can}}}

\newcommand{\supp}{{\operatorname{supp}}}

\newcommand*\angles[1]{\langle #1 \rangle}
\newcommand*\prnths[1]{\left( #1 \right)}

\newcommand*\N{\mathbb{N}}
\newcommand*\Z{\mathbb{Z}}

\newcommand*\R{\mathbb{R}}

\newcommand*\G{\Gamma}

\newcommand*\bq{\textbf{q}}
\newcommand*\eD{\epsilon_D}

\newcommand{\Cm}{\mathrm{C}}
\renewcommand{\Im}{\mathrm{I}}
\newcommand{\Jm}{\mathrm{J}}
\newcommand{\Lm}{\mathrm{L}}
\newcommand{\Mm}{\mathrm{M}}
\newcommand{\Zm}{\mathrm{Z}}

\newcommand*\varhrulefill[1][3pt]{\leavevmode\leaders\hrule height#1\hfill\kern0pt}

\let \origcite \cite
\def \cite[#1]#2{\textbf{\origcite[#1]{#2}}}

\makeatletter
\renewcommand{\section}{\@startsection {section}{1}{\z@}
	{-3.5ex \@plus -1ex \@minus -.2ex}
	{2.3ex \@plus .2ex}
	{\Large \bfseries \filcenter}}
\makeatother

\definecolor{lightyellow}{RGB}{255, 255, 197}

\title{\bf{Computation of Admissible Arakelov--Green \\ Functions on Metrized Graphs}}

\author{Ruben Merlijn van Dijk and Enis Kaya}

\begin{document}

\maketitle

\begin{abstract}
Metrized graphs are nonarchimedean analogues of Riemann surfaces, and Arakelov--Green functions on these graphs are of fundamental importance for some aspects of arithmetic geometry. In the present paper, we give an explicit formula for an admissible Arakelov--Green function on a metrized graph, extending Cinkir's formula for the canonical Arakelov--Green function. Based on our formula, we present and implement an algorithm in the computer algebra system \texttt{SageMath} for explicitly computing such functions. We illustrate our algorithm with computational examples.
\end{abstract}

\setcounter{tocdepth}{1}
\tableofcontents

\section{Introduction}

Metrized graphs, also referred to as metric graphs, are finite and connected graphs whose edges are viewed as closed line segments. This definition allows one to define a distance function, and therefore to do analysis, on metrized graphs. In other words, metrized graphs are not just combinatorial objects, but also analytic objects.

Metrized graphs appear in many areas of mathematics such as arithmetic geometry \cite[]{RumelyBook,ChinburgRumely,Zhang1993}, tropical geometry \cite[]{BakerFaber11,ABBR15I,ABBR15II,BakerJensen2016}, algebra \cite[]{Vogtmann2002} and analysis \cite[]{BakerRumely}. Apart from mathematics, they also appear, sometimes with different names, in several areas of science such as physics, chemistry, biology and engineering; see, for instance,  \cite[Section~1.9]{BakerRumely} and the references therein. Our motivation for the present paper comes from arithmetic and tropical geometry. In these contexts, the main examples of metrized graphs arise from the dual graphs associated to the special fibres of semistable models of algebraic curves. 

\bigskip

In \cite[]{RumelyBook,ChinburgRumely,Zhang1993}, metrized graphs were introduced as nonarchimedean analogues of Riemann surfaces. We, in particular, have \textit{Arakelov--Green functions} on such graphs. These functions play key roles in many articles such as \cite[]{RumelyBook,ChinburgRumely,Zhang1993,Moriwaki3,Moriwaki2,BakerRumely,Zhang2010,CinkirEffBog,deJong2018,deJong2019,Wilms19,Xinyi21,LooperSilvermanWilms,AminiNicolussiI,AminiNicolussiII,AminiNicolussiIII}. Therefore, it is desirable to have an efficient method for computing Arakelov--Green functions on metrized graphs. The first step in this direction was taken by Cinkir in \cite[]{Cinkir2014}: he gave an explicit formula for the \textit{canonical} Arakelov--Green function, which led to an efficient algorithm for explicit computations. There are two main goals of the current paper\footnote{This paper is based on the first-named author's BSc thesis \cite[]{vanDijkThesis}, which was (unofficially) supervised by the second-named author.}:
\begin{enumerate}
    \item Firstly, we extend Cinkir's formula in order to deal with any \textit{admissible} Arakelov--Green function, of which the canonical Arakelov--Green function is a special case.
    \item Secondly, we devise an efficient algorithm for computing admissible Arakelov--Green functions using our formula, and provide its implementation in \texttt{SageMath} \cite[]{Sage}.
\end{enumerate}
Let us make this more precise.

Each well-behaved measure $\mu$ on a metrized graph $\G$ gives rise to an Arakelov--Green function $g_{\mu}$ on $\G\times\G$. The canonical Arakelov--Green function $g_{\mu_\can}$ is the one coming from the \textit{canonical} measure $\mu_\can$ discovered by Chinburg and Rumely in \cite[]{ChinburgRumely}. Baker and Rumely \cite[]{BakerRumely} showed that
\[g_{\mu_\can}(x,y) = \tau(\G) - \frac{1}{2}r(x,y),\ \ \ x,y\in\G\]
where $\tau(\G)$ is the \textit{tau constant} of $\G$, a fundamental invariant of $\G$, and $r(x,y)$ is the \textit{resistance function} on $\G$. Thanks to \cite[]{Cinkir2011}, computing $\tau(\G)$ is quite easy. The function $r(x,y)$ (and thus $g_{\mu_\can}(x,y)$) was made explicit by Cinkir in \cite[]{Cinkir2014} using electric circuit theory.

On the other hand, the admissible Arakelov--Green function $g_{\mu_D}$ is the one coming from the \textit{admissible} measure\footnote{This is usually called the admissible metric in the literature. Our terminology is not standard, but it seems to be convenient.} $\mu_D$ constructed by Zhang in \cite[]{Zhang1993}, where $D$ is a divisor on $\G$. Setting $D = 0$ recovers the ``canonical'' setting, so $g_{\mu_{D}}$ is a generalization of $g_{\mu_{\can}}$. In this paper, we make the function $g_{\mu_{D}}$ explicit. To do so, we closely follow Cinkir's strategy. We first introduce a function $\tau_D$ on $\G\times\G$; this is what we call the \textit{tau function} of $\G$ (\cref{def:TauFun}). We then prove that
\[g_{\mu_D}(x,y) = \tau_D(x,y) - \frac{1}{2}r(x,y),\ \ \ x,y\in\G\]
(\cref{prop:gD_tauD}). Then we make the function $\tau_D(x,y)$ explicit (\cref{prop:kappa}). And finally, making use of the explicit expression for $r(x,y)$ provided by Cinkir in \cite[]{Cinkir2014}, we give an explicit formula for $g_{\mu_{D}}$ (\cref{thm:MainThm1} and \cref{thm:MainThm2}).

Our explicit formula allows us to give an efficient algorithm for computing admissible Arakelov--Green functions $g_{\mu_D}$ for every divisor $D$ (Algorithms~\ref{alg:ConnectivityMatrix}-\ref{alg:rD}-\ref{alg:ValueMatrix}). As an application, we give an algorithm to compute the \textit{epsilon invariant} $\eD$ (\cref{alg:Epsilon}). This constant appears in effective Bogomolov's conjectures and is defined in terms of the function $g_{\mu_D}$. We furthermore provide an implementation of these algorithms in \texttt{SageMath}. Our code can be found at 
\begin{equation}\label{eq:code}
\text{\url{https://github.com/KayaEnis/ArakelovGreen.git}}.
\end{equation}
Finally, we provide a number of examples that illustrate our algorithms (\cref{sec:CompExa}).

\subsection{Outline and notation}

A brief outline of the paper is as follows. \cref{sec:Prel} is devoted to some of the basic definitions in the theory of metrized graphs and a summary of a few essential results of Cinkir \cite[]{Cinkir2014}. In \cref{sec:ArGrFc}, after giving the definition of the Arakelov--Green function $g_{\mu_D}$, we derive an explicit formula for it. Our algorithms, in particular the one for computing the Arakelov--Green function $g_{\mu_D}$, are presented in \cref{sec:Alg}. Computational examples are given in \cref{sec:CompExa}.

\bigskip

For the convenience of the reader, we give the following table which gives the notations for the central objects of the paper:
\begin{center}
	\begin{tabular}{rcl}
	    $\G$         &--&  A metrized graph.	 \\
		$V(\G)$  	 &--&  A vertex set for $\G$.  \\
		$E(\G)$		&--&  The corresponding edge set for $\G$. \\
		$\Lm$ &--&  The discrete Laplacian matrix of $\G$. \\
		$\Lm^+$ &--&  The Moore--Penrose generalized inverse of $\Lm$. \\
		$j$ &--&  The voltage function on $\G$.\\
		$r$ &--&  The resistance function on $\G$. \\
		$\Cm$ &--&  The connectivity matrix of $\G$. \\
		$\mu_\can$  &--&  The canonical measure on $\G$. \\
		$g_{\mu_\can}$  &--&  The Arakelov--Green function corresponding to $\mu_\can$.\\
		$\tau(\G)$  &--&  The tau constant of $\G$. \\
		$D$  &--&  A divisor on $\G$. \\
		$\mu_D$  &--&  The admissible measure on $\G$ with respect to $D$.\\
		$g_{\mu_D}$  &--& The Arakelov--Green function corresponding to $\mu_D$. \\
		$\tau_D$  &--&  The tau function of $\G$ with respect to $D$. \\
		$\Zm_D$  &--&  The value matrix of $g_{\mu_D}$. 
	\end{tabular}
\end{center}

\begin{center} \textbf{Acknowledgements} \end{center}

We would like to express our thanks to Steffen M{\"u}ller for his encouragement and support during our work on this project. Besides, we thank him, Omid Amini, Robin de Jong and Oliver Lorscheid for their helpful comments on an early draft.

\section{Preliminaries}
\label{sec:Prel}

In this section, we introduce notations and assumptions used throughout this paper and collect some auxiliary results. We will mostly follow Cinkir \cite[]{Cinkir2014}. See also \cite[]{BakerFaber,BakerRumely,CinkirThesis,BakerRumelyBook}.

\subsection{Metrized graphs}

There are various equivalent definitions for metrized graphs. We will be using the following.

\begin{definition}{\bf{[Metrized graph]}}
A \textit{metrized graph} is a finite and connected graph $\G$ such that each of its edges is equipped with a parametrization. 
\end{definition}

Let $\G$ be a metrized graph. The \textit{valence} of a point $p$ on $\G$ is defined as the number of directions emanating from $p$ and denoted by $v(p)$. Let $V(\G)$ be a vertex set for $\G$ such that
\begin{enumerate}
    \item $V(\G)$ is nonempty and finite, and
    \item any point on $\G$ of valence different from $2$ belongs to $V(\G)$.
\end{enumerate}
We remark that $V(\G)$ is not unique, and when needed, we can enlarge it by considering some of the points of valence $2$ as vertices; this flexibility will play an important role in the sequel. The corresponding set of edges of $\G$, denoted by $E(\G)$, is defined as the set of closed line segments with endpoints in $V(\G)$. Every edge has a length, and the \textit{total length} of $\G$ is defined to be
\[\ell(\G) \coloneqq \sum_{e\in E(\G)} (\text{length of } e) \in \R_{>0}.\]

For an edge $e\in E(\G)$, the graph obtained from $\G$ by deleting the interior of $e$ is denoted by $\G - e$. Note that this graph may not be connected. This motivates the following definition.

\begin{definition}{\bf{[Bridge]}}
\label{def:Bridge}
An edge $e\in E(\G)$ is called a \textit{bridge} if the graph $\G - e$ is disconnected. For a bridge $e$ with endpoints $p$ and $q$, we will denote the subgraph of $\G - e$ containing $p$ (resp. $q$) by $\G_p$ (resp. $\G_q$).
\end{definition}

Let us look at an example. Consider the metrized graph $\G$ depicted in \cref{fig:CircleLine}. It has three vertices, $p_1,p_2,p_3$, and three edges of lengths $a,b,c$. The total length of $\G$ is $\ell(\G) = a + b + c$. There is only one bridge, namely the edge connecting $p_2$ and $p_3$.

\begin{figure}[h]
	\centering
	\begin{tikzpicture}
		\def \radius {1cm}

        \filldraw (-1,0) circle (2pt);
        \filldraw (1,0) circle (2pt);
        \filldraw (3,0) circle (2pt);

		\node at ({180}:1.4\radius) {$p_1$};
		\node at ({360}:.6\radius) {$p_2$};
		\node at ({360}:3.4\radius) {$p_3$};

		\draw ({0}:\radius) arc ({0}:{180}:\radius) node[midway,above]{$a$};
		\draw ({180}:\radius) arc ({180}:{360}:\radius) node[midway,below]{$b$};
		\draw (1,0) -- (3,0);
	\end{tikzpicture}
	\caption{A metrized graph consisting of a circle and a line segment.}\label{fig:CircleLine}
\end{figure}
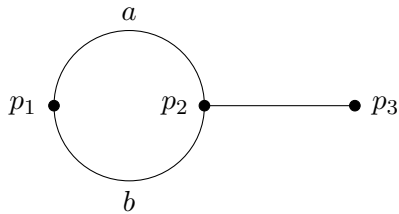

From now on, we assume that the vertex set $V(\G)$ is \textit{adequate} (sometimes called \textit{optimal}). That is to say, there are neither parallel edges nor loops in our graph $\G$. Notice that this can be achieved quite easily by enlarging the vertex set suitably, if necessary. For instance, the vertex set $V(\G) = \{p_1,p_2,p_3\}$ for the graph $\G$ in \cref{fig:CircleLine} is not adequate, since there are two parallel edges. An adequate vertex set can be obtained by adjoining one additional point $p_0$ to the given vertex set; see Figure \ref{fig:CircleLineAdeq}. Notice that the total length is unchanged.

\begin{figure}[h]
	\centering
		\begin{tikzpicture}
			\node[vertex, label=above:$p_0$] (0) at (0,2) {};
			\node[vertex, label=left:$p_1$] (1) at (-2,0) {};
			\node[vertex, label=below:$p_2$] (2) at (2,0) {};
			\node[vertex, label=right:$p_3$] (3) at (5,0) {};

			\draw [postaction=decorate] (0) -- (1) node[pos=.4,left=0.2]{$\frac{a}{2}$};
			\draw [postaction=decorate] (0) -- (2) node[pos=.4,right=0.2]{$\frac{a}{2}$};
			\draw [postaction=decorate] (1) -- (2) node[pos=.5,below=0.1]{$b$};
			\draw [postaction=decorate] (2) -- (3) node[pos=.5,below=0.1]{$c$};
		\end{tikzpicture}
	\caption{The metrized graph in Figure \ref{fig:CircleLine} with an adequate vertex set.}\label{fig:CircleLineAdeq}
\end{figure}
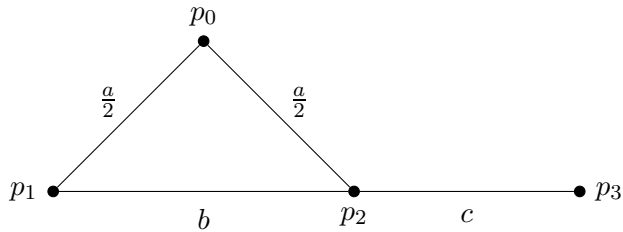

We now define the \textit{discrete Laplacian matrix} of $\G$. Fix an ordering of the vertex set $V(\G)$, and let $n$ be the number of vertices.

\begin{definition}{\bf{[Discrete Laplacian matrix]}}
The matrix $\Lm = (l_{pq})_{n\times n}$, where
\[l_{pq} = 
\begin{dcases}
\omit\hfil$0$\hfil &\text{if $p\neq q$, and they are not connected by an edge}, \\
\omit\hfil$-\ \dfrac{1}{\text{length of } e}$\hfil &\text{if $p\neq q$, and  they are connected by an edge $e$}, \\
-\sum_{s\in V(\G)\setminus\{p\}} l_{ps} &\text{if $p=q$},
\end{dcases}\]
is the
\textit{discrete Laplacian matrix}\footnote{Since the vertex set is adequate, there can be at most one edge between two distinct points; therefore, the discrete Laplacian matrix is well-defined.} of $\G$.
\end{definition}

The matrix $\Lm$ is not invertible, but has the Moore--Penrose generalized inverse $\Lm^{+} = (l_{pq}^+)_{n\times n}$, which can be easily computed via the formula
\begin{equation}\label{eq:Lplus}
\Lm^{+} = \prnths{\Lm - \frac{1}{n}\Jm_n}^{-1} + \frac{1}{n}\Jm_n
\end{equation}
where $\Jm_n$ is the $n\times n$ matrix whose entries are all $1$; see \cite[Chapter~10]{RaoMitra}. We remark that $\Lm^{+}$ is the unique matrix such that
\[ \Lm\Lm^{+}\Lm = \Lm, \ \ \ \Lm^{+}\Lm\Lm^{+} = \Lm^{+}, \ \ \ (\Lm\Lm^{+})^T = \Lm\Lm^{+} \ \ \ \text{and} \ \ \ (\Lm^{+}\Lm)^T = \Lm^{+}\Lm.\]
Here, $\Mm^T$ denotes the transpose of a matrix $\Mm$.

\subsection{Voltage and resistance functions on metrized graphs}

Let $\G$ be a metrized graph with an adequate vertex set $V(\G)$. In \cite[Section~2]{ChinburgRumely}, Chinburg and Rumely defined and studied a function $j_z(x,y)$ on $\G$, which gives a fundamental solution of the \textit{Laplacian operator}, and enjoys the following key properties:
\begin{enumerate}
    \item It is jointly continuous as a function of $x,y$ and $z$.
    \item It is symmetric in $x$ and $y$; that is, $j_z(x,y) = j_z(y,x)$ for all $x,y,z\in\G$.
    \item It is nonnegative with $j_x(x,y) = j_y(x,y) = 0$ for all $x,y\in\G$.
\end{enumerate}
See also \cite[Section~6]{BakerFaber} and \cite[Section~6]{BakerRumely}. This function is called the \textit{voltage function} on $\G$ because of its physical interpretation; see, for example, \cite[page~762]{Cinkir2014} and \cite[Section~1.5]{BakerRumely}.

\begin{definition}{\bf{[Resistance function]}}
The function
\[r(x,y) \coloneqq j_y(x,x),\ \ \ x,y\in\G \]
is called the \textit{(effective) resistance function} on $\G$.
\end{definition}

It is quite easy to compute the voltage and resistance functions at vertices of $\G$, as shown in the following lemma.

\begin{lemma}\emph{(\cite[Lemmas~3.4 and 3.5]{CinkirFoster})}
\label{lem:jr_vertices}
Fix an ordering of $V(\G)$ and let $n$ be the number of vertices. Denote the Moore--Penrose generalized inverse of the discrete Laplacian matrix of $\G$ by $\Lm^{+} = (l_{pq}^+)_{n\times n}$. Then, for any $p,q,s\in V(\G)$, we have
\[j_s(p,q) = l_{ss}^{+} - l_{sp}^{+} - l_{sq}^{+} + l_{pq}^{+}\ \ \text{and} \ \ r(p,q) = l_{pp}^{+} - 2l_{pq}^{+} + l_{qq}^{+}.\]
\end{lemma}

The following two lemmas are due to Cinkir, and they give a method to compute the resistance function at more general points on $\G$. In the equalities below, for points $x$ and $y$ in $\G$, we see algebraic expressions in $x$ and $y$ such as $xy$, $x-y$ and $\left|x-y\right|$. In these expressions, by $x$ and $y$ we mean the real numbers corresponding to the points $x$ and $y$ on $\G$ via the parametrization. For instance, if $x$ belongs to an edge $e$ of length $L$, then under the identification of $e$ with the closed interval $[0,L]$, we have $x\in [0,L]$. This convention will be used in the rest of the paper.

\begin{lemma}\emph{(\cite[Lemma~3.1 and Remark~3.4(a)]{Cinkir2014})}
\label{lem:rxy_oneedge}
Let $x,y\in \G$. If both $x$ and $y$ belong to the same edge $e$ of length $L$ with endpoints $p$ and $q$, then
\[r(x,y) = \left|x-y\right| - (x-y)^2\frac{L - r(p,q)}{L^2}.\]
Moreover, if $e$ is a bridge, then
\[r(x,y) = \left|x-y\right|.\]
\end{lemma}

\begin{lemma}\emph{(\cite[Theorem~3.3 and Remark~3.4(c)(d)]{Cinkir2014})}
\label{lem:rxy_twoedges}
Let $x,y\in \G$ belong to edges $e_i$ and $e_j$ respectively, and $e_i\neq e_j$. Call the lengths $L_i, L_j$ and the endpoints $p_i,q_i,p_j,q_j$. We then have
\begin{align*}
r(x,y) = &-x^2\frac{L_i - r(p_i,q_i)}{L_i^2} - y^2\frac{L_j - r(p_j,q_j)}{L_j^2} + \frac{2xy}{L_iL_j}\prnths{j_{p_j}(p_i,q_j) - j_{p_j}(q_i,q_j)}\\ 
&+ \frac{x}{L_i}\prnths{L_i - 2j_{p_i}(q_i,p_j)} + \frac{y}{L_j}\prnths{L_j - 2j_{p_j}(p_i,q_j)} + r(p_i,p_j). 
\end{align*}
Now suppose at least one of $e_i$ and $e_j$ is a bridge.
\begin{enumerate}
	\item If both $e_i$ and $e_j$ are bridges, then
	\[r(x,y) = 
	\begin{dcases}
	\omit\hfil$ x + y + r(p_i,p_j)$\hfil &\text{if $p_i$ and $p_j$ are between $x$ and $y$}, \\
	\omit\hfil$ x - y + L_j + r(p_i,q_j)$\hfil &\text{if $p_i$ and $q_j$ are between $x$ and $y$}, \\
	\omit\hfil$ - x + y + L_i + r(q_i,p_j)$\hfil &\text{if $q_i$ and $p_j$ are between $x$ and $y$}, \\
	- x - y + L_i + L_j + r(q_i,q_j) &\text{if $q_i$ and $q_j$ are between $x$ and $y$}.
	\end{dcases}\]
	
	\item If only $e_i$ is a bridge (the other case can be handled in a completely analogous way), then
	\[r(x,y) = 
	\begin{dcases}
	\omit\hfil$-y^2\dfrac{L_j - r(p_j,q_j)}{L_j^2} + T_1$\hfil &\text{if $y\in \G_{p_i}$}, \\
	\omit\hfil$-y^2\dfrac{L_j - r(p_j,q_j)}{L_j^2} + T_2$\hfil &\text{if $y\in \G_{q_i}$},
	\end{dcases}\]
	where
	\begin{equation}
	\begin{aligned}\label{eq:T1T2}
	T_1 &= y\frac{L_j - r(p_j,q_j) + r(p_i,q_j) - r(p_i,p_j)}{L_j} + x + r(p_i,p_j),\\
	T_2 &= y\frac{L_j - r(p_j,q_j) + r(q_i,q_j) - r(q_i,p_j)}{L_j} - x + L_i + r(q_i,p_j).
	\end{aligned}
	\end{equation}
\end{enumerate}
\end{lemma}

\bigskip

We also record, for later use, the following lemma.

\begin{lemma}\label{lem:CinkirLemma3.2}
\emph{(\cite[Lemma~3.2 and Remark~3.4(b)]{Cinkir2014})} 
Let $x\in \G$ belong to an edge $e$ of length $L$ with endpoints $p$ and $q$. For any $s\in V(\G)$, we have
\[r(s,x) = -x^2\frac{L - r(p,q)}{L^2} + x\frac{L - r(p,q) + r(s,q) - r(s,p)}{L} + r(s,p).\]
Moreover, if $e$ is a bridge, then
\[r(s,x) = 
\begin{dcases}
\omit\hfil$x + r(s,p)$\hfil &\text{if $s\in \G_p$}, \\
- x + L + r(s,q) &\text{if $s\in \G_q$}.
\end{dcases}\]
\end{lemma}

\subsection{Divisors on metrized graphs}

Let $\G$ be a metrized graph with a vertex set $V(\G)$. A \textit{divisor} $D$ on $\G$ is a finite formal sum of the form 
\[D = \sum_{s\in\G} a_s s\]
where $a_s\in\Z$ for each $s$. Such a divisor is called \textit{effective} if $a_s \geq 0$ for all $s$. Moreover, its \textit{degree} and \textit{support} are defined as
\[\deg(D) \coloneqq \sum_{s\in \G} a_s \ \ \text{and} \ \ \supp(D)\coloneqq \{s\in\G \mid a_s\neq 0\},\]
respectively.

Note that, for any divisor $D$ on $\G$, we can equip $\G$ with a vertex set $V(\G)$ with the property that $\supp(D)\subset V(\G)$. This allows us to express the divisor $D$ in terms of vertices, rather than arbitrary points on $\G$.

\bigskip

Let us conclude this section by introducing \textit{polarized metrized graphs} and their \textit{canonical divisors}; we will need these notions in \cref{subsec:ExBanana}. The following is taken from \cite[Section~4]{CinkirEffBog} (see also \cite[]{Zhang2010,Faber09,Cinkir2015MathComp,Cinkir2015}).

\begin{definition}{\bf{[Canonical divisor \& polarized metrized graph]}}
\label{def:PolMetGr}
Let $\bq\colon \G\to\Z$ be a function supported on $V(\G)$. The divisor
\[K \coloneqq \sum_{s\in V(\G)} \prnths{v(s) - 2 + 2\bq(s)}s\]
is called the \textit{canonical divisor} of $(\G,\bq)$.
The pair $(\G,\bq)$ is called a \textit{polarized metrized graph} if $\bq$ is nonnegative and $K$ is effective.
\end{definition}

\section{Arakelov--Green functions}
\label{sec:ArGrFc}

Our main goal in this section is to provide an explicit formula for the Arakelov--Green function $g_{\mu_D}$ on a metrized graph $\G$, extending Cinkir's results \cite[Theorems~4.3 and 4.4]{Cinkir2014}. Here, the function $g_{\mu_D}$ is defined with respect to the measure $\mu_D$, where $D$ is a divisor on $\G$.  

Let $V(\G)$ be an adequate vertex set for the metrized graph $\G$, and write $E(\G)$ for the corresponding edge set. 

\subsection{Arakelov--Green function $g_{\mu_D}$}

\begin{definition}{\bf{[Arakelov--Green function]}}
\label{def:ArGrFc}
Let $\mu$ a real-valued signed Borel measure on $\G$ such that $\mu(\G)$ is $1$ and $|\mu|(\G)$ is finite, and set
\[j_{\mu}(x,y) \coloneqq \int_{\G} j_z(x,y)d\mu(z),\ \ \ x,y\in\G.\]
The corresponding \textit{Arakelov--Green function} is defined as 
\[g_{\mu}(x,y) \coloneqq j_{\mu}(x,y) - \int_{\G^3} j_z(x,y)d\mu(z)d\mu(x)d\mu(y),\ \ \ x,y\in\G.\]
Notice that the latter term is a constant which depends only on $\G$ and $\mu$.
\end{definition}
It should be remarked that the function $j_{\mu}$ (and hence $g_{\mu}$) is symmetric and jointly continuous in both variables. 

\bigskip

Chinburg and Rumely showed that there exists a unique measure $\mu$ such that the restriction of $j_{\mu}(x,y)$ to the diagonal in $\G\times\G$ is constant; see \cite[Theorem~2.11]{ChinburgRumely}. This measure is called the \textit{canonical measure} and denoted by $\mu_\can$.

\begin{definition}{\bf{[Tau constant]}} 
The \textit{tau constant} $\tau(\G)$ of $\G$ is defined as $\frac{1}{2}j_{\mu_\can}(x,x)$ for some (hence for all) $x\in\G$.
\end{definition}

The tau constant $\tau(\G)$ is a fundamental invariant of $\G$. Thanks to the following proposition, one can easily compute it using the discrete Laplacian matrix. This fact will play a key role in our computations.

\begin{proposition}\emph{(\cite[Theorem~1.1]{Cinkir2011})}
\label{prop:tau_formula}
Fix an ordering of $V(\G)$ and let $n$ be the number of vertices. Let $\Lm = (l_{pq})_{n\times n}$ be the discrete Laplacian matrix of $\G$, and denote its Moore--Penrose generalized inverse by $\Lm^+ = (l^+_{pq})_{n\times n}$. We then have
\[\tau(\G) = -\frac{1}{12} \sum_{e_i \in E(\G)} l_{p_iq_i}\prnths{\frac{1}{l_{p_iq_i}} + l^+_{p_ip_i} - 2l^+_{p_iq_i} + l^+_{q_iq_i}}^2
+ \frac{1}{4} \sum_{q,s\in V(\G)} l_{qs}l^+_{qq}l^+_{ss} 
+ \frac{1}{n}\operatorname{trace}\left(L^+\right)\]
where we denote the end points of an edge $e_i$ by $p_i$ and $q_i$.
\end{proposition}

The following result says that the \textit{canonical} Arakelov--Green function $g_{\mu_\can}$, the one corresponding to the canonical measure ${\mu_\can}$, is essentially the resistance function.

\begin{proposition}\label{prop:gcan_tau}
\emph{(\cite[Theorem~14.1(3)]{BakerRumely})}
For all $x,y\in\G$, we have 
\[g_{\mu_\can}(x,y) = \tau(\G) - \frac{1}{2}r(x,y).\]
\end{proposition}

Combining this proposition with Lemmas~\ref{lem:rxy_oneedge} and \ref{lem:rxy_twoedges}, Cinkir gave an explicit formula for the Arakelov--Green function $g_{\mu_\can}$; see \cite[Theorems~4.3 and 4.4]{Cinkir2014}. 

\bigskip

We now fix a divisor $D = \sum_{s\in V(\G)} a_s s$ on $\G$ whose degree is different from $-2$. Zhang showed that there exists a unique measure $\mu_D$ on $\G$ with the property that $g_{\mu_D}(x,D) + g_{\mu_D}(x,x)$ is constant on $\G$, where 
\[g_{\mu_D}(x,D) \coloneqq \sum_{s\in V(\G)} a_s g_{\mu_D}(x,s);\]
see \cite[Theorem~3.2]{Zhang1993}. This measure is called the \textit{admissible measure on $\G$ with respect to $D$}. Thanks to \cite[Section~4.4]{CinkirThesis}, we have
\[\mu_D(x) = \frac{1}{\deg(D) + 2}\prnths{\sum_{s\in V(\G)} a_s \delta_s(x) + 2\mu_\can(x)}\]
where $\delta_s$ is the Dirac measure, and the corresponding \textit{admissible} Arakelov--Green function $g_{\mu_{D}}$ can be expressed as
\begin{equation}\label{eq:g_muD}
g_{\mu_{D}}(x,y) = \frac{1}{\deg(D) + 2}\prnths{\sum_{s\in V(\G)} a_sj_s(x,y) + 4\tau(\G) - r(x,y)} - c_{\mu_D}
\end{equation}
where
\begin{equation}\label{eq:CmuD}
 c_{\mu_D} \coloneqq \frac{1}{2(\deg(D)+2)^2}\prnths{8\tau(\G)\prnths{\deg(D) + 1} + \sum_{s,t\in V(\G)} a_s a_t r(s,t)}.  
\end{equation}
See also \cite[Section~5]{Cinkir2014}. Notice that, when $D = 0$, we have
\[\mu_D = \mu_\can\ \ \ \text{and} \ \ \ g_{\mu_{D}} = g_{\mu_{\can}}.\]
In other words, $g_{\mu_{D}}$ is a generalization of $g_{\mu_{\can}}$. In the next subsection, we will extend \cite[Theorems~4.3 and 4.4]{Cinkir2014} to a formula for $g_{\mu_{D}}$.

\subsection{An explicit formula for $g_{\mu_{D}}$}
\label{subsec:ExpFormg_muD}

We keep the notation of the previous subsection. In particular, we have a divisor $D = \sum_{s\in V(\G)} a_s s$ on $\G$ such that $\deg(D)\neq -2$. We begin with a definition.

\begin{definition}{\bf{[Tau function]}}
\label{def:TauFun}
The \textit{tau function of $\G$ with respect to $D$} is defined to be
\begin{equation}\label{eq:chiD}
\tau_D(x,y) \coloneqq \frac{1}{\deg(D) + 2}\prnths{4\tau(\G) + \frac{1}{2}\big(r(D,x) + r(D,y)\big)} - c_{\mu_D}, \ \ \ x,y\in \G
\end{equation}
where $c_{\mu_D}$ as in \cref{eq:CmuD}, and
\begin{equation}\label{eq:kappaD}
r(D,x) \coloneqq \sum_{s\in V(\G)} a_s r(s,x),\ \ \ x\in\G.
\end{equation}
\end{definition}

Clearly, the tau function $\tau_D(x,y)$ is nothing but the tau constant $\tau(\G)$ when $D = 0$. With this in mind, the following proposition plays the role of \cref{prop:gcan_tau}.

\begin{proposition}
\label{prop:gD_tauD}
For all $x,y\in \G$, we have
\[g_{\mu_{D}}(x,y) = \tau_D(x,y) - \frac{1}{2}r(x,y).\]
\end{proposition}

\begin{proof}
This follows easily from by plugging the relation
\[j_s(x,y) = \frac{1}{2}\prnths{r(s,x) + r(s,y) - r(x,y)},\ \ \ s,x,y\in\G\]
(see \cite[Remark~3.5]{Cinkir2014} or \cite[Equation~(2)]{CinkirFoster}) into \cref{eq:g_muD} and rearranging the terms.
\end{proof}

As in the ``canonical'' case, we will combine this result with Lemmas~\ref{lem:rxy_oneedge} and \ref{lem:rxy_twoedges} to give an explicit formula for the Arakelov--Green function $g_{\mu_D}$. There is an intermediate step to be taken, however: we should make the term $\tau_D(x,y)$ explicit, for which we need to make the function $r(D,x)$ explicit. This is done in the following proposition.

\begin{proposition}
\label{prop:kappa}
Let $x\in \G$ belong to an edge $e$ of length $L$ with endpoints $p$ and $q$. We then have
\[r(D,x) = \sum_{s\in V(\G)} a_s \prnths{-x^2\frac{L - r(p,q)}{L^2} + x\frac{L - r(p,q) + r(s,q) - r(s,p)}{L} + r(s,p)}.\]
Moreover, if $e$ is a bridge, then
\[r(D,x) = \sum_{s\in V(\G)\cap \G_p} a_s \prnths{x + r(s,p)} + \sum_{s\in V(\G)\cap \G_q} a_s \prnths{- x + L + r(s,q)}.\]
\end{proposition}

\begin{proof}
This is a direct consequence of \cref{lem:CinkirLemma3.2}.
\end{proof}

After this preparatory step, we can now state the main theoretical results of the current paper.

\begin{theorem}\label{thm:MainThm1}
{\bf{[Main Theorem~I]}} 
Let $x,y\in \G$. If both $x$ and $y$ belong to the same edge $e$ of length $L$ with endpoints $p$ and $q$, and $e$ is not a bridge, then
\[g_{\mu_{D}}(x,y) = \tau_D(x,y) - \dfrac{1}{2}\left|x-y\right| + (x-y)^2\dfrac{L - r(p,q)}{2L^2}.\]
Now suppose $x$ and $y$ belong to edges $e_i$ and $e_j$ respectively, $e_i\neq e_j$, and neither $e_i$ nor $e_j$ is a bridge. Call the lengths $L_i, L_j$ and the endpoints $p_i,q_i,p_j,q_j$. Then
\begin{align*}
g_{\mu_{D}}(x,y) = \tau_D(x,y) &+ x^2\frac{L_i - r(p_i,q_i)}{2L_i^2} + y^2\frac{L_j - r(p_j,q_j)}{2L_j^2}\\ 
 &- \frac{xy}{L_iL_j}\prnths{j_{p_j}(p_i,q_j) - j_{p_j}(q_i,q_j)} - \frac{x}{2L_i}\prnths{L_i - 2j_{p_i}(q_i,p_j)}\\
 &- \frac{y}{2L_j}\prnths{L_j - 2j_{p_j}(p_i,q_j)} - \frac{1}{2}r(p_i,p_j).
\end{align*}
\end{theorem}

\begin{theorem}\label{thm:MainThm2}
{\bf{[Main Theorem~II]}} 
Let $x,y\in \G$. If both $x$ and $y$ belong to the same edge $e$ of length $L$ with endpoints $p$ and $q$, and $e$ is a bridge, then
\[g_{\mu_{D}}(x,y) = \tau_D(x,y) - \dfrac{1}{2}\left|x-y\right|.\]
Now suppose $x$ and $y$ belong to edges $e_i$ and $e_j$ respectively, $e_i\neq e_j$, and at least one of $e_i$ and $e_j$ is a bridge. Call the lengths $L_i, L_j$ and the endpoints $p_i,q_i,p_j,q_j$.
\begin{enumerate}
    \item If both $e_i$ and $e_j$ are bridges, then
    \[g_{\mu_{D}}(x,y) = 
    \begin{dcases}
    \omit\hfil$\tau_D(x,y)- \dfrac{1}{2}\prnths{x + y + r(p_i,p_j)}$\hfil &\text{if $p_i$ and $p_j$ are between $x$ and $y$}, \\
    \omit\hfil$\tau_D(x,y)- \dfrac{1}{2}\prnths{x - y + L_j + r(p_i,q_j)}$\hfil &\text{if $p_i$ and $q_j$ are between $x$ and $y$}, \\
    \omit\hfil$\tau_D(x,y)- \dfrac{1}{2}\prnths{- x + y + L_i + r(q_i,p_j)}$\hfil &\text{if $q_i$ and $p_j$ are between $x$ and $y$}, \\
    \tau_D(x,y)- \dfrac{1}{2}\prnths{- x - y + L_i + L_j + r(q_i,q_j)} &\text{if $q_i$ and $q_j$ are between $x$ and $y$}.
    \end{dcases}\]
    
    \item If only $e_i$ is a bridge (the other case can be handled in a completely analogous way), then
    \[g_{\mu_{D}}(x,y) = 
    \begin{dcases}
    \omit\hfil$\tau_D(x,y)+ y^2\dfrac{L_j - r(p_j,q_j)}{2L_j^2} - \dfrac{1}{2}T_1$\hfil &\text{if $y\in \G_{p_i}$}, \\
    \omit\hfil$\tau_D(x,y)+ y^2\dfrac{L_j - r(p_j,q_j)}{2L_j^2} - \dfrac{1}{2}T_2$\hfil &\text{if $y\in \G_{q_i}$},
    \end{dcases}\]
    where $T_1$ and $T_2$ are as in \cref{eq:T1T2}.
    \end{enumerate}
\end{theorem}

\bigskip

\begin{proof}[Proof of Theorems~\ref{thm:MainThm1}-\ref{thm:MainThm2}] 
The proofs follow from \cref{prop:gD_tauD}, \cref{lem:rxy_oneedge}  and  \cref{lem:rxy_twoedges}.
\end{proof}

\bigskip

We conclude this section with a definition. It is clear that $g_{\mu_{D}}$ is a piecewisely defined function on each pair of edges. Cinkir represented the function $g_{\can}$ by, what he called, a \textit{value matrix}, see \cite[page~769]{Cinkir2014}. We will do the same:

\begin{definition}{\bf{[Value matrix]}}
\label{def:ValueMatrix}
Enumerate the edges as $(e_1,e_2,\dots,e_m)$, and for $i,j\in\{1,2,\dots,m\}$, define 
\[z_{ij}\colon e_i\times e_j\to \R,\ \ \ (x,y)\mapsto g_{\mu_{D}}(x,y).\]
We call the matrix $\Zm_D = (z_{ij})_{m\times m}$ the \textit{value matrix} of $g_{\mu_{D}}$.
\end{definition} 

Note that $\Zm_D$ is a matrix of functions.

\begin{remark}\label{rem:SymOfZD}
The value matrix $\Zm_D = (z_{ij})_{m\times m}$ is symmetric in the sense that
\[z_{ij}(x,y) = z_{ji}(y,x),\ \ \ (x,y)\in e_i\times e_j.\]
\end{remark}

\section{Algorithms}
\label{sec:Alg}

In this section, we present our algorithms. Our main algorithm, namely \cref{alg:ValueMatrix}, computes the Arakelov--Green function $g_{\mu_D}$ using the formulas provided in \cref{thm:MainThm1} and \cref{thm:MainThm2}. As an application, we provide an algorithm for computing the \textit{epsilon invariant} $\eD$, an important constant whose definition relies on $g_{\mu_D}$; see \cref{alg:Epsilon}. The algorithms described in this section were implemented by the authors in \texttt{SageMath}. Our code can be found at \eqref{eq:code}.

\bigskip

Fix a metrized graph $\G$ with an adequate vertex set $V(\G) = \{p_0,p_1,\dots,p_{n-1}\}$; we write the corresponding edge set as $E(\G) = \{e_0,e_1,\dots,e_{m-1}\}$\footnote{In this and the next section, lists will be $0$-indexed, because \texttt{SageMath} uses $0$-based indexing.}, and for each edge $e_i$, we set $L_i\coloneqq \text{length of } e_i$. Let $p_{e_i}$ (resp. $q_{e_i}$) be the endpoint of $e_i$ that corresponds to $0$ (resp. $L_i$) under the identification of $e_i$ with the interval $[0,L_i]$.

Recall that if $e_i$ is a bridge, then $\G_{p_i}$ (resp. $\G_{q_i}$) denotes the subgraph of $\G - e_i$ containing the endpoint $p_i$ (resp. $q_i$).

\subsection{Basic computations}

With the data we have, we can easily compute the discrete Laplacian matrix $\Lm$ of $\G$ and its Moore--Penrose generalized inverse $\Lm^+$ (recall the formula in \cref{eq:Lplus}). These matrices allow us to compute the voltage and resistance functions at vertices of $\G$ using \cref{lem:jr_vertices}, and the tau constant $\tau(\G)$ using \cref{prop:tau_formula}.

\subsection{Connectivity matrix}

An essential object of our algorithms is the \textit{connectivity matrix}, which encodes all information about the structure of $\G$ that we need in order to apply our main theorems. To define it, we need two auxiliary functions: $\alpha$ and $\beta$.

Let us fix two distinct edges $e_i, e_j$, and suppose $e_i$ is a bridge. The first auxiliary function, $\alpha$, encodes whether $e_j \subset \G_{p_{e_i}}$ or $e_j \subset \G_{q_{e_i}}$ with a single bit:
\begin{equation*}
	\alpha(e_i,e_j) \coloneqq
	\begin{dcases}
		0 &\text{if } e_j \subset \G_{p_{e_i}},\\
		1 &\text{if } e_j \subset \G_{q_{e_i}}.
	\end{dcases}
\end{equation*}

Now suppose $e_j$ is also a bridge. The second auxiliary function, $\beta$, requires the following notion.

\begin{definition}{\bf{[Closest neighbours]}}
The \textit{closest neighbours} of $e_i$ and $e_j$, denoted by $\angles{e_i,e_j}$, is the unique element $(x_i,x_j)$ in $e_i\times e_j$ such that
\[d(x_i,x_j) = \min_{(y_i,y_j)\in e_i\times e_j} d(y_i,y_j),\]
where $d(x,y)$ is the distance between two points $x,y$ on $\G$, by which we mean the length of the shortest path between them.
\end{definition}

There are not so many options for $\angles{e_i,e_j}$. More precisely, we have
\[\angles{e_i,e_j} \in \{(p_{e_i},p_{e_j}),(p_{e_i},q_{e_j}),(q_{e_i},p_{e_j}),(q_{e_i},q_{e_j})\}.\]
We illustrate this by way of an example.

\begin{figure}[t]
	\centering
	\begin{tikzpicture}[decoration={markings,mark = at position 0.55 with {\arrow[thick]{latex}}}] 
		\node[vertex, label=left:$p_0$] (1) at (-3,0) {};
		\node[vertex, label=below:$p_1$] (2) at (-1,0) {};
		\node[vertex, label=below:$p_2$] (3) at (1,-1.5) {};
		\node[vertex, label=above:$p_3$] (4) at (1,1.5) {};
		\node[vertex, label=below:$p_4$] (5) at (3,0) {};
		\node[vertex, label=right:$p_5$] (6) at (5,0) {};

		\draw[postaction=decorate] (1) -- (2) node[midway, above=0.15]{$e_0$};
		\draw[postaction=decorate] (2) -- (3) node[pos=0.7, left=0.2]{$e_1$};
		\draw[postaction=decorate] (2) -- (4) node[pos=0.7, left=0.2]{$e_2$};
		\draw[postaction=decorate] (3) -- (5) node[pos=0.3, right=0.2]{$e_3$};
		\draw[postaction=decorate] (4) -- (5) node[pos=0.3, right=0.2]{$e_4$};
		\draw[postaction=decorate] (5) -- (6) node[midway, above=0.15]{$e_5$};
	\end{tikzpicture}
	\caption{A metrized graph with two bridges.}\label{fig:ExClosNeigh}
\end{figure}
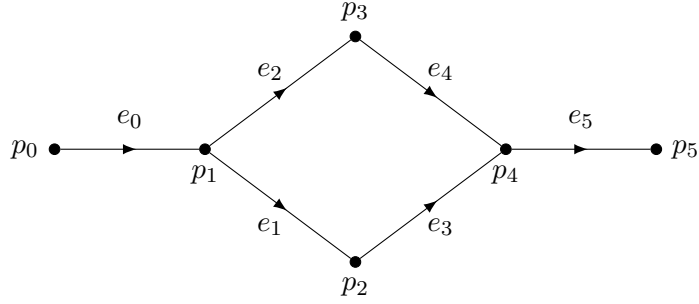

\begin{example}
Consider the metrized graph in Figure \ref{fig:ExClosNeigh}, where
\begin{align*}
p_0 &= p_{e_0},     &  p_1 &= q_{e_0} = p_{e_1} = p_{e_2},     &  p_2 &= q_{e_1} = p_{e_3},\\
p_3 &= q_{e_2} = p_{e_4},  &  p_4 &= q_{e_3} = q_{e_4} = p_{e_5},  &  p_5 &= q_{e_5}.
\end{align*}
The edges $e_0$ and $e_5$ are bridges, and the closest neighbours is given by
\[\angles{e_0,e_5} = (p_1,p_4).\]
\end{example}

Returning to our discussion, the following function encodes the closest neighbours of $e_i$ and $e_j$ with two bits:
\begin{equation*}
	\beta(e_i,e_j) \coloneqq
	\begin{dcases}
		\omit\hfil$0$\hfil &\text{if } \angles{e_i,e_j} = (p_{e_i},p_{e_j}),\\
		\omit\hfil$1$\hfil &\text{if } \angles{e_i,e_j} = (p_{e_i},q_{e_j}),\\
		10 &\text{if } \angles{e_i,e_j} = (q_{e_i},p_{e_j}),\\
		11 &\text{if } \angles{e_i,e_j} = (q_{e_i},q_{e_j}).
	\end{dcases}
\end{equation*}

We can finally define the main object of this subsection.

\begin{definition}{\bf{[Connectivity matrix]}}
The \textit{connectivity matrix} $\Cm = (c_{ij})_{m \times m}$ of $\G$ is a matrix with binary entries defined as follows: 
\begin{equation*}
	c_{ij} \coloneqq 
	\begin{dcases}
		\omit\hfil$0$\hfil & \text{if neither $e_i$ nor $e_j$ is a bridge},\\
		\omit\hfil$1$\hfil & \text{if $i=j$ and $e_i$ is a bridge},\\
		\omit\hfil$\alpha(e_i,e_j)$\hfil & \text{if $e_i$ is a bridge but $e_j$ is not},\\
		\omit\hfil$\alpha(e_j,e_i)$\hfil & \text{if $e_j$ is a bridge but $e_i$ is not},\\
		100\alpha(e_i,e_j)+\beta(e_i,e_j) & \text{if $i\neq j$, and both $e_i$ and $e_j$ are bridges}.
	\end{dcases}
\end{equation*}    
\end{definition}

Summarizing, given two edges $e_i, e_j \in E(\G)$, the entries $c_{ii}$ and $c_{jj}$ tell us which of them are bridges. If one of them, say $e_i$, is a bridge and the other is not, the single bit in $c_{ij}$ encodes whether $e_j$ is part of the subgraph $\G_{p_{e_i}}$ or $\G_{q_{e_i}}$. If $i\neq j$ and both $e_i$ and $e_j$ are bridges, so $c_{ii} = c_{jj} = 1$, then $c_{ij}$ consists of three bits: the first bit tells us if $e_j \subset \G_{p_{e_i}}$ or $e_j \subset \G_{q_{e_i}}$, while the second and third bits encode the closest neighbours of $e_i$ and $e_j$. Here is an example.

\begin{example}
Consider once again the metrized graph in \cref{fig:ExClosNeigh}. As we have six edges, the connectivity matrix $\Cm$ will be a $6 \times 6$ matrix. 

The edges $e_0$ and $e_5$ are bridges, so $c_{00} = c_{55} = 1$, while the other diagonal entries are $0$. Now, we need to determine $\alpha(e_0,e_i)$ for $i = 1,2,3,4,5$ and $\alpha(e_5,e_i)$ for $i = 0,1,2,3,4$. We have $e_1 \subset \G_{p_1}$, and since $p_1 = q_{e_0}$, we have $\alpha(e_0,e_1) = 1$, so $c_{01} = c_{10} = 1$. Similarly, $e_1 \subset \G_{p_4}$, so $\alpha(e_5,e_1) = 0$, and thus $c_{51} = c_{15} = 0$. We can exploit this symmetry for any pair of edges with one bridge. The entries $c_{05}$ and $c_{50}$ are more complicated. We have that $e_5 \subset \G_{p_1}$, so $\alpha(e_0,e_5) = 1$. The closest neighbours of the bridges are given by $\angles{e_0,e_5} = (p_1,p_4)$, so $\beta(e_0,e_5) = 10$. This yields
\[c_{05} = 100\alpha(e_0,e_5) + \beta(e_0,e_5) = 110.\]
Computing all $c_{ij}$'s gives us the matrix
\begin{equation*}
	\Cm = \begin{bmatrix}
		1 & 1 & 1 & 1 & 1 & 110\\
		1 & 0 & 0 & 0 & 0 & 0\\
		1 & 0 & 0 & 0 & 0 & 0\\
		1 & 0 & 0 & 0 & 0 & 0\\
		1 & 0 & 0 & 0 & 0 & 0\\
		1 & 0 & 0 & 0 & 0 & 1\\
	\end{bmatrix}.
\end{equation*}
We can now quickly look up aspects of the metrized graph's structure: since $c_{00} = 1$ and $c_{33} = 0$, we know that $e_0$ is a bridge but $e_3$ is not. Since $c_{03} = 1$, we can also tell that $e_3 \subset \G_{q_{e_0}}$.
\end{example}

\cref{alg:ConnectivityMatrix} describes how to generate the connectivity matrix.

\bigskip

\begin{algorithm}[H]
	\caption{\bf Generating the connectivity matrix}
	\label{alg:ConnectivityMatrix} 
	\BlankLine
	\KwIn{Metrized graph $\G$ with $V(\G) = \{p_0,p_1,\dots,p_{n-1}\}$ and $E(\G) = \{e_0,e_1,\dots,e_{m-1}\}$.}
	\BlankLine
	\KwOut{The connectivity matrix $\Cm = (c_{ij})$.}

	\BlankLine
	$\Cm \coloneqq 0_{m\times m}$;
	\BlankLine

	\For(){$i \in \{0,1,\dots,m-1\}$}{
		\If{$e_i$ is a bridge}{
			$c_{ii} \leftarrow 1$\;
		}
	}

	\BlankLine

	\For(){$i \in \{0,1,\dots,m-1\}$}{
		\uIf(){$c_{ii} = 1$}{
			\For(){$j \in \{i+1,i+2,\dots,m-1\}$}{
				$c_{ij} \leftarrow \alpha(e_i,e_j)$;

				\uIf(){$c_{jj} = 1$}{
					$c_{ij} \leftarrow 100c_{ij} + \beta(e_i,e_j)$\;
					$c_{ji} \leftarrow 100\alpha(e_j,e_i) + \beta(e_j,e_i)$\;
				}\Else(){
					$c_{ji} \leftarrow c_{ij}$;
				}
				
			}
		}\Else(){
			\For(){$j \in \{i+1,i+2,\dots,m-1\}$}{
				\If(){$c_{jj} = 1$}{
					$c_{ij} \leftarrow \alpha(e_j,e_i)$\;
					$c_{ji} \leftarrow   c_{ij}$\;
				}
			}
		}
	}
\end{algorithm}

\bigskip

For the rest of this section, we fix a divisor $D$ on $\G$ with $\supp(D) \subset V(\G)$. We can then express it as $D = \sum^{n-1}_{k=0}a_kp_k$ for some $a_k\in \Z$, since $V(\G) = \{p_0,p_1,\dots,p_{n-1}\}$. In fact, we can (and will) treat such divisors as vectors of $n$ entries:
\[D = (a_0,a_1,\dots,a_{n-1})\in \Z^n.\]

\subsection{Tau function}
\label{subsec:TauFun}

As discussed in \cref{subsec:ExpFormg_muD}, in order to make $g_{\mu_D}$ explicit, we should make the tau function $\tau_D$ explicit. The constant
\[c_{\mu_D} = \frac{1}{2\prnths{\deg(D)+2}^2} \prnths{8\tau(\G)(\deg(D)+1) + \sum^{n-1}_{k,l=0} a_ka_lr(p_k,p_l)}\]
can easily be computed by making use of \cref{lem:jr_vertices}. Therefore, it suffices to make $r(D,\cdot)$ explicit. Using \cref{prop:kappa}, \cref{alg:rD} computes $r(D,\cdot)$ on a given edge $e_i$. More precisely, it returns a function $r_i(D,\cdot)\colon e_i \to \mathbb{R}$ such that 
\[r_i(D,x) = r(D,x), \ \ \ x \in e_i.\] 

\bigskip

\begin{algorithm}[H]
	\caption{\bf Computing the function $r(D,\cdot)$}
	\label{alg:rD}

	\SetKw{Break}{break}
	\BlankLine
	\KwIn{
	\begin{itemize}
	    \item Metrized graph $\G$ with $V(\G) = \{p_0,p_1,\dots,p_{n-1}\}$ and $E(\G) = \{e_0,e_1,\dots,e_{m-1}\}$,
	    \item Divisor $D = (a_0,a_1,\dots,a_{n-1})$ on $\G$,
	    \item Integer $i \in \{0,1,\dots,m-1\}$.
	\end{itemize}
	}

	\KwOut{Function $r_i(D,\cdot)$ which satisfies $r_i(D,x) = r(D,x)$ for $x \in e_i$.}
	
	\BlankLine
	Compute the connectivity matrix $\Cm = (c_{ij})$ using \cref{alg:ConnectivityMatrix};

    \BlankLine
	$r_i(D,x) := 0$;
	\BlankLine

	\uIf(){$c_{ii} = 0$}{
		\For{$k \in \{0,1,\dots,n-1\}$}{
			$\begin{aligned}
				r_i(D,x) \leftarrow r_i(D,x) + a_k \bigg(&-x^2 \dfrac{L_i - r(p_{e_i},q_{e_i})}{L_i^2}\\
				& + x \dfrac{L_i - r(p_{e_i},q_{e_i}) + r(p_k,q_{e_i}) - r(p_k,p_{e_i})}{L_i} + r(p_k,p_{e_i})\bigg);
			\end{aligned}$
		}
	}\Else(){
		\For(){$k \in \{0,1,\dots,n-1\}$}{
			\uIf{$p_k = p_{e_i}$}{
				$r_i(D,x) \leftarrow r_i(D,x) + a_k x$\;
			}\uElseIf{$p_k = q_{e_i}$}{
				$r_i(D,x) \leftarrow r_i(D,x) + a_k \prnths{- x + L_i}$\;
			}\Else(){
				\For(){$j \in \{0,1,\dots,m-1\}\setminus \{i\}$}{
					\If(){$p_k$ is an end point of $e_j$}{
						\uIf{($c_{jj} = 0$ and $c_{ij} = 0$) or ($c_{jj} = 1$ and $c_{ij} < 100$)}{
							$r_i(D,x) \leftarrow r_i(D,x) + a_k \prnths{x + r(p_k,p_{e_i})}$;

						}\Else{
							$r_i(D,x) \leftarrow r_i(D,x) + a_k \prnths{- x + L_i + r(p_k,q_{e_i})}$;
						}
						\Break\;
					}
				}

			}
		}
	}
\end{algorithm}

\bigskip

Repeating this algorithm for each edge yields an $m$-vector $\big(r_i(D,\cdot)\big)_{i=1,\dots,m}$, which completely describes $r(D,\cdot)$ on $\G$. Then, using this vector, we can calculate
\[(\tau_{D})_{i,j}(x,y) \coloneqq \frac{1}{\deg(D) + 2} \prnths{4\tau(\G) + \frac{1}{2}\big(r_i(D,x)+r_j(D,y)\big)} - c_{\mu_D}\]
for any $i,j\in\{0,1,\dots,m-1\}$.

\subsection{Value matrix}

We can now compute the value matrix $\Zm_D = (z_{ij})$ of $g_{\mu_{D}}$ using the formulas in Theorems~\ref{thm:MainThm1} and \ref{thm:MainThm2}. The process of computing $z_{ij}$ is described in \cref{alg:ValueMatrix}. Repeating this algorithm for each $i$ and $j$ yields the matrix $\Zm_D$.

In \cref{alg:ValueMatrix}, we use the shorthands
\[\begin{aligned}
    T_{e_i,1} &\coloneqq y\frac{L_j - r(p_{e_j},q_{e_j}) + r(p_{e_i},q_{e_j}) - r(p_{e_i},p_{e_j})}{L_j} + x + r(p_{e_i},p_{e_j}),\\
    T_{e_i,2} &\coloneqq y\frac{L_j - r(p_{e_j},q_{e_j}) + r(q_{e_i},q_{e_j}) - r(q_{e_i},p_{e_j})}{L_j} - x  + L_i + r(q_{e_i},p_{e_j}),\\
    T_{e_j,1} &\coloneqq x\frac{L_i - r(p_{e_i},q_{e_i}) + r(p_{e_j},q_{e_i}) - r(p_{e_j},p_{e_i})}{L_i} + y + r(p_{e_j},p_{e_i}),\\
    T_{e_j,2} &\coloneqq x\frac{L_i - r(p_{e_i},q_{e_i}) + r(q_{e_j},q_{e_i}) - r(q_{e_j},p_{e_i})}{L_i} - y + L_j + r(q_{e_j},p_{e_i}).
\end{aligned}\]
Recall that the connectivity matrix $\Cm = (c_{ij})$ is a matrix with binary entries.

\bigskip

\begin{algorithm}[H]
	\caption{\bf Computing the matrix $\Zm_D = (z_{ij})$}
	\label{alg:ValueMatrix}

    \small

    \BlankLine
	\KwIn{
	\begin{itemize}
	    \item Metrized graph $\G$ with $V(\G) = \{p_0,p_1,\dots,p_{n-1}\}$ and $E(\G) = \{e_0,e_1,\dots,e_{m-1}\}$,
	    \item Divisor $D$ on $\G$, 
	    \item Integers $i,j \in \{0,1,\dots,m-1\}$.
	\end{itemize}
	}
	
	\KwOut{Function $z_{ij}$ which satisfies $z_{ij}(x,y) = g_{\mu_D}(x,y)$ for $x \in e_i$ and $y \in e_j$.}
	
	\BlankLine
	Compute the connectivity matrix $\Cm = (c_{ij})$ using \cref{alg:ConnectivityMatrix};
	
	\BlankLine
	Compute the function $(\tau_{D})_{i,j}$ as in \cref{subsec:TauFun};

	\BlankLine
	\uIf(){$i = j$}{
		\uIf(){$c_{ii} = 0$}{
			$z_{ii}(x,y) \leftarrow (\tau_{D})_{i,j}(x,y) - \dfrac{1}{2}\left|x-y\right| + \prnths{x-y}^2\dfrac{L_i-r(p_{e_i},q_{e_i})}{2L_i^2};$
		}\Else(){	
		$z_{ii}(x,y) \leftarrow (\tau_{D})_{i,j}(x,y) - \dfrac{1}{2}\left|x-y\right|$\;
		}
	}
	\Else(){
		\uIf(){$c_{ii}=0$ and $c_{jj}=0$}{
			$\begin{aligned}
				z_{ij}(x,y) \leftarrow (\tau_{D})_{i,j}(x,y) &+ x^2 \dfrac{L_i - r(p_{e_i},q_{e_i})}{2L_i^2} + y^2 \dfrac{L_j - r(p_{e_j},q_{e_j})}{2L_j^2}\\ 
				&- \dfrac{xy}{L_iL_j} \prnths{j_{p_{e_j}}(p_{e_i},q_{e_j}) - j_{p_{e_j}}(q_{e_i},q_{e_j})} -\dfrac{x}{2L_i}\prnths{L_i - 2j_{p_{e_i}}(q_{e_i},p_{e_j})} \\
				&- \dfrac{y}{2L_j}\prnths{L_j - 2j_{p_{e_j}}(p_{e_i},q_{e_j})} -\dfrac{1}{2}r(p_{e_i},p_{e_j});
			\end{aligned}$
			}
			\uElseIf(){$c_{ii} = 1$ and $c_{jj} = 0$}{
				\uIf(){$c_{ij} = 0$}{
					$z_{ij}(x,y) \leftarrow (\tau_{D})_{i,j}(x,y) + y^2 \dfrac{L_j - r(p_{e_j},q_{e_j})}{2L_j^2} - \dfrac{1}{2}T_{e_i,1}$;
				} \Else(){
					$z_{ij}(x,y) \leftarrow (\tau_{D})_{i,j}(x,y) + y^2 \dfrac{L_j - r(p_{e_j},q_{e_j})}{2L_j^2} - \dfrac{1}{2}T_{e_i,2};$
				}
				}
				\uElseIf(){$c_{ii} = 0$ and $c_{jj} = 1$}{
					\uIf(){$c_{ij} = 0$}{
						$z_{ij}(x,y) \leftarrow (\tau_{D})_{i,j}(x,y) + x^2 \dfrac{L_i - r(p_{e_i},q_{e_i})}{2L_i^2} - \dfrac{1}{2}T_{e_j,1}$;

					} \Else(){
						$z_{ij}(x,y) \leftarrow (\tau_{D})_{i,j}(x,y) + x^2 \dfrac{L_i - r(p_{e_i},q_{e_i})}{2L_i^2} - \dfrac{1}{2}T_{e_j,2}$;
					}

					}	\Else(){
						\uCase(){$c_{ij} \bmod 100 \equiv 0$}{
							$z_{ij}(x,y) \leftarrow (\tau_{D})_{i,j}(x,y)-\frac{1}{2}\prnths{x + y + r(p_{e_i},p_{e_j})}$\;
						}\uCase(){$c_{ij} \bmod100\equiv 1$}{
							$z_{ij}(x,y) \leftarrow (\tau_{D})_{i,j}(x,y) - \frac{1}{2}\prnths{x - y + L_j +  r(p_{e_i},q_{e_j})}$\;
						}\uCase(){$c_{ij} \bmod100\equiv 10$}{
							$z_{ij}(x,y) \leftarrow (\tau_{D})_{i,j}(x,y)- \frac{1}{2}\prnths{- x + y + L_i + r(q_{e_i},p_{e_j})}$\;
						}\Case(){$c_{ij} \bmod100\equiv  11$}{
							$z_{ij}(x,y) \leftarrow (\tau_{D})_{i,j}(x,y)- \frac{1}{2}\prnths{- x - y + L_i + L_j + r(q_{e_i},q_{e_j})}$\;
						}
					}
					}
\end{algorithm}

\subsection{Application: epsilon invariant}
\label{subsec:EpsInv}

In \cite[]{Moriwaki1,Moriwaki3,Moriwaki2,Moriwaki4}, Moriwaki studied a constant associated to $D$, denoted by $\eD$, that appears in effective Bogomolov's conjectures. Thanks to \cite[Lemma~4.1]{Moriwaki1}, we have a nice formula for $\eD$: for any point $p$ on $\G$, we have
\begin{equation}\label{eq:EpsInvMor}
\eD = \prnths{\deg(D) + 2} \sum_{k=0}^{n-1} a_kg_{\mu_D}(p,p_{k}) + \sum_{k=0}^{n-1} a_kr(p,p_{k})
\end{equation}
(recall that $D = \sum^{n-1}_{k=0}a_kp_k$).

\begin{definition}{\bf{[Epsilon invariant]}}
The constant $\eD$ is called the \textit{epsilon invariant} of $\G$ with respect to $D$.
\end{definition}

As an application of our algorithms, we give a method to compute $\eD$; see \cref{alg:Epsilon}. Our method exploits the fact that the formula in \cref{eq:EpsInvMor} is independent of the choice of $p$; in the algorithm, we choose $p=p_{e_0}$.

\bigskip

\begin{algorithm}[H]
	\caption{\bf Computing the invariant $\eD$}
	\label{alg:Epsilon}

	\SetKw{Break}{break}
	
	\BlankLine
	\KwIn{
	\begin{itemize}
	    \item Metrized graph $\G$ with $V(\G) = \{p_0,p_1,\dots,p_{n-1}\}$ and $E(\G) = \{e_0,e_1,\dots,e_{m-1}\}$,
	    \item Divisor $D = (a_0,a_1,\dots,a_{n-1})$ on $\G$.
	\end{itemize}
	}
	\KwOut{The invariant $\eD$.}
	
	\BlankLine
	Compute the value matrix $\Zm_D = (z_{ij})$ using \cref{alg:ValueMatrix};

	\BlankLine
	
	$\eD \coloneqq 0$;
	\BlankLine

	\For(){$k \in \{0,1,\dots,n-1\}$}{
		\For(){$l \in \{0,1,\dots,m-1\}$}{
			\uIf(){$p_k = p_{e_l}$}{
				$g_k \leftarrow z_{0l}(p_{e_0},p_{e_l})$\;
				\Break\;
			}

			\ElseIf(){$p_k = q_{e_l}$}{
				$g_k \leftarrow z_{0l}(p_{e_0},q_{e_l})$\;
				\Break\;
			}

		}
			$\eD \leftarrow \eD + a_k g_k$\;
	}

    \BlankLine

	$\eD \leftarrow (\deg(D)+2)\eD$\;
	
	\BlankLine

	$\eD \leftarrow \eD + \sum^{n-1}_{k=0} a_kr(p_{e_0},p_k)$\;
	
	\BlankLine
\end{algorithm}

\bigskip

In \cite[Theorem~4.27]{CinkirThesis}, Cinkir found an alternative expression for the invariant $\eD$, which does not involve the function $g_{\mu_D}$ and is easier to compute:
\begin{equation}\label{eq:EpsInvCin}
\eD = \frac{1}{\deg(D)+2}\prnths{4\tau(\G)\deg(D) + \sum^{n-1}_{k,l=0} a_ka_lr(p_k,p_l)}.
\end{equation}

In our implementation, we offer two different ways of computing $\eD$. One is based on \cref{alg:Epsilon}, and the second on considering \cref{eq:EpsInvCin}.

\subsection{Consistency checks}

In order to test the correctness of our algorithms, we do two quick experiments.

\begin{enumerate}
    \item Firstly, vertices of valency of greater than one can be represented in different ways, but their images under the entries of the value matrix should be independent of the choices of representations. More concretely, let $p$ and $q$ be two vertices of $\G$, and suppose $v(p) > 1$. Then, there should be at least two edges connected to $p$, call them $e_i$ and $e_j$. If the point $q$ belongs to an edge $e_k$, then we must have
    \[z_{ik}(p,q) = g_{\mu_D}(p,q) = z_{jk}(p,q)\]
    (recall \cref{def:ValueMatrix}). This test can be generalized to check any pair of vertices at least one of which has valence at least two. An script for this test is given in the file \texttt{FirstTest.sage}.
    \item Secondly, recall the formula for $g_{\mu_D}$ in \cref{eq:g_muD}. Notice that, using \cref{lem:jr_vertices}, one can immediately compute the function $g_{\mu_D}$ on vertices. Of course, these values should match the values computed with \cref{alg:ValueMatrix}. A script that runs this test on all vertices of a given metrized graph is provided in \texttt{SecondTest.sage}.
\end{enumerate}

\section{Computational examples}
\label{sec:CompExa}

In this final section, we present some examples computed using our implementation.

\subsection{Circle graph}

In this example, we compute the value matrix of $g_{\mu_\can}$, which is $g_{\mu_D}$ for $D=0$, on the metrized graph depicted in \cref{fig:CircleGraph}. 

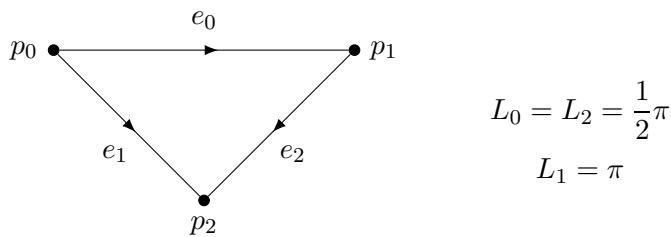
\begin{figure}[h]
	\centering
	\begin{tikzpicture}[decoration={markings,mark = at position 0.55 with {\arrow[thick]{latex}}}] 
		\node[vertex, label=left: $p_0$] (1) at (-2,2) {};
		\node[vertex, label=right:$p_1$] (2) at (2,2)  {};
		\node[vertex, label=below:$p_2$] (3) at (0,0)  {};
		
		\draw[postaction=decorate] (1) -- (2) node[midway, above=0.15]{$e_0$};
		\draw[postaction=decorate] (1) -- (3) node[pos=0.7, left=0.25]{$e_1$};
		\draw[postaction=decorate] (2) -- (3) node[pos=0.7, right=0.25]{$e_2$};
		
		\qquad
		
		\filldraw (5,1.2) node{$L_0 = L_2 = \dfrac{1}{2}\pi$};
		\filldraw (5,0.4) node{$L_1 = \pi$};
	\end{tikzpicture}
	\caption{A circle graph.}\label{fig:CircleGraph}
\end{figure}

Enumerate the vertices as $(p_1,p_0,p_2)$ and the edges as $(e_0,e_1,e_2)$. Then the discrete Laplacian matrix and its Moore--Penrose generalized inverse are
\begin{equation*}
	\Lm = \frac{1}{\pi}
	\begin{bmatrix}
		4 & -2 & -2 \\
		-2 & 3 & -1 \\
		-2 & -1 & 3
	\end{bmatrix}
	\quad\text{and}\quad
	\Lm^+ = \frac{\pi}{72} 
	\begin{bmatrix}
		8 & -4 & -4 \\
		-4 & 11 & -7 \\
		-4 & -7 & 11
	\end{bmatrix},
\end{equation*}
and the tau constant is $\tau(\G) = \frac{\pi}{6}$. The connectivity matrix is the $3\times3$ zero matrix since the metrized graph has no bridges. The value matrix is then given by $\Zm_D = (z_{ij})_{0\leq i,j\leq 2}$, where
\begin{align*}
z_{00}(x,y) & = \frac{1}{\pi}\prnths{\frac{1}{4}(x-y)^2 - \frac{\pi}{2}\left|x-y\right| + \frac{\pi^2}{6}}, \\ 
z_{01}(x,y) & = \frac{1}{\pi}\prnths{\frac{1}{4}(x+y)^2 - \frac{\pi}{2}(x+y) + \frac{\pi^2}{6}}, \\
z_{02}(x,y) & = \frac{1}{\pi}\prnths{\frac{1}{4}(x-y)^2 + \frac{\pi}{4}(x-y) - \frac{\pi^2}{48}}, \\
z_{10}(x,y) & = \frac{1}{\pi}\prnths{\frac{1}{4}(x+y)^2 - \frac{\pi}{2}(x+y) + \frac{\pi^2}{6}}, \\
z_{11}(x,y) & = \frac{1}{\pi}\prnths{\frac{1}{4}(x-y)^2 - \frac{\pi}{2}\left|x-y\right| + \frac{\pi^2}{6}}, \\ 
z_{12}(x,y) & = \frac{1}{\pi}\prnths{\frac{1}{4}(x+y)^2 - \frac{\pi}{4}(x+y) - \frac{\pi^2}{48}}, \\
z_{20}(x,y) & = \frac{1}{\pi}\prnths{\frac{1}{4}(x-y)^2 - \frac{\pi}{4}(x-y) - \frac{\pi^2}{48}}, \\
z_{21}(x,y) & = \frac{1}{\pi}\prnths{\frac{1}{4}(x+y)^2 - \frac{\pi}{4}(x+y) - \frac{\pi^2}{48}}, \\
z_{22}(x,y) & = \frac{1}{\pi}\prnths{\frac{1}{4}(x-y)^2 - \frac{\pi}{2}\left|x-y\right| + \frac{\pi^2}{6}}.
\end{align*}

The value matrix of $g_{\mu_\can}$ was already computed by Cinkir in \cite[Example~6.1]{Cinkir2014}. Our matrix is the same as his matrix except one entry, namely $z_{20}$. Our matrix is symmetric in the sense of \cref{rem:SymOfZD}, but his matrix is not. Therefore, his matrix contains a tiny error\footnote{His matrix is symmetric in the classical sense, which might have caused the error.}. The complete computation can be found in the file \texttt{Ex1.sage}. 

\subsection{Joint circles}

Let $C_1$ and $C_2$ denote two circles of lengths $\ell_1$ and $\ell_2$ respectively, which touch each other externally at a point $p_0$; see \cref{fig:JointCircles}. Define the metrized graph $\G$ to be the circles' union, and let a divisor on $\G$ be given by $D = 2p_0$. Our goal in this example is to compute the value matrix $\Zm_D$ of $g_{\mu_D}$.

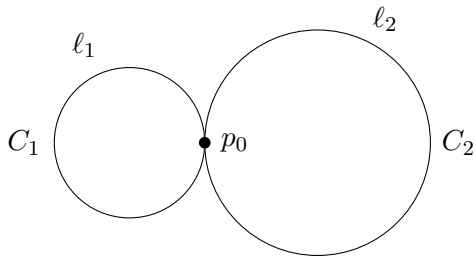
\begin{figure}[h]
	\centering
	\begin{tikzpicture}
		\def \radiusone {1cm}
		\def \radiustwo {1.5cm}
		\draw (-\radiusone,0) circle (\radiusone);
		\draw (\radiustwo,0) circle (\radiustwo);
		\filldraw (0,0) circle (2pt);
		\node at (.4,0) {$p_0$};
		\node at (-2.4*\radiusone,0) {$C_1$};
		\node at (2.25*\radiustwo,0) {$C_2$};
		\node at (-1.6,1.3) {$\ell_1$};
		\node at (2.4,1.65) {$\ell_2$};
	\end{tikzpicture}
	\caption{A metrized graph consisting of two circles joint at $p_0$.}\label{fig:JointCircles}
\end{figure}

Clearly, the given vertex set $V(\G) = \{p_0\}$ is not adequate, but an adequate vertex set can easily be obtained by adjoining two additional points on each circle to the given vertex set. The new metrized graph is illustrated in Figure \ref{fig:JointCirclesAdeq}.

\begin{figure}[h]
	\centering
	\begin{tikzpicture}[decoration={markings,mark = at position 0.55 with {\arrow[thick]{latex}}}] 
		\node[vertex, label=above:$p_0$] (0) at (0,0) {};
		\node[vertex, label=left:$p_1$] (1) at (-2,-1) {};
		\node[vertex, label=left:$p_2$] (2) at (-2,1) {};
		\node[vertex, label=right:$p_3$] (3) at (2,-1) {};
		\node[vertex, label=right:$p_4$] (4) at (2,1) {};

		\draw [postaction=decorate] (0) -- (1) node[pos=.4,below=0.2]{$e_0$};
		\draw [postaction=decorate] (1) -- (2) node[pos=.5,left=0.1]{$e_1$};
		\draw [postaction=decorate] (2) -- (0) node[pos=.6,above=0.2]{$e_2$};

		\draw [postaction=decorate] (0) -- (3) node[pos=.4,below=0.2]{$e_3$};
		\draw [postaction=decorate] (3) -- (4) node[pos=.5,right=0.1]{$e_4$};
		\draw [postaction=decorate] (4) -- (0) node[pos=.6,above=0.2]{$e_5$};
		
		\qquad
		
		\filldraw (6,0.6) node{$L_0=L_1=L_2=\dfrac{\ell_1}{3}$};
		\filldraw (6,-0.6) node{$L_3=L_4=L_5=\dfrac{\ell_2}{3}$};
	\end{tikzpicture}
	\caption{The metrized graph in \cref{fig:JointCircles} with an adequate vertex set.}\label{fig:JointCirclesAdeq}
\end{figure}

The discrete Laplacian matrix is
\begin{equation*}
	\Lm = \frac{3}{\ell_1\ell_2}
	\begin{bmatrix}
		2\ell_1 + 2\ell_2 & -\ell_2 & -\ell_2 & -\ell_1 & -\ell_1\\
		-\ell_2 & 2\ell_2 & -\ell_2 & 0 & 0 \\
		-\ell_2 & -\ell_2 & 2\ell_2 & 0 & 0\\
		-\ell_1 & 0 & 0 & 2\ell_1 & -\ell_1\\
		-\ell_1 & 0 & 0 & -\ell_1 & 2\ell_1
	\end{bmatrix},
\end{equation*}
and its Moore--Penrose generalized inverse is
\begin{equation*}
	\Lm^+ = \frac{\ell_1}{225}
	\begin{bmatrix}
		6 & -9 & -9 & 6 & 6 \\
		-9 & 26 & 1 & -9 & -9 \\
		-9 & 1 & 26 & -9 & -9 \\
		6 & -9 & -9 & 6 & 6 \\
		6 & -9 & -9 & 6 & 6
	\end{bmatrix}
	+ \frac{\ell_2}{225}
	\begin{bmatrix}
		6 & 6 & 6 & -9 & -9 \\
		6 & 6 & 6 & -9 & -9 \\
		6 & 6 & 6 & -9 & -9 \\
		-9 & -9 & -9 & 26 & 1 \\
		-9 & -9 & -9 & 1 & 26
	\end{bmatrix}.
\end{equation*}
The tau constant is $\tau(\G) = \frac{1}{12}(\ell_1+\ell_2)$. The divisor $D$ is $(2,0,0,0,0)$, and the value matrix of $g_{\mu_D}$ is
\[\Zm_D = \begin{bmatrix} M_1(x,y) & W(x,y) \\ W^T(y,x) & M_2(x,y)\end{bmatrix},\]
where
\begin{equation*}
	\begin{aligned}
		M_i(x,y) = &\ \frac{1}{4\ell_i}\prnths{x^2+y^2-4xy}\Jm_3 + \frac{1}{12}
	\begin{bmatrix}
		3x+3y  & 5x - y & x + y \\
		-x + 5y & x+y  & 3x - 3y \\
		x + y & -3x + 3y & -x -y
	\end{bmatrix}\\
	&\ - \frac{1}{2}\left|x-y\right|\Im_3 + \frac{\ell_j}{48} \Jm_3 + \frac{\ell_i}{144}
	\begin{bmatrix}
		3 & -5& -5\\
		-5 & 19& 3\\
		-5& 3& 19
	\end{bmatrix}
	\end{aligned}
\end{equation*}
with $j \in \{1,2\} \setminus \{i\}$, and
\begin{equation*}
	\begin{aligned}
		W(x,y) = &\ \frac{1}{4\ell_1\ell_2}\prnths{\ell_2x^2+\ell_1y^2}\Jm_3 -  \frac{1}{12}
		\begin{bmatrix}
			3x + 3y & 3x + y & 3x - y\\
			x + 3y & x + y & x - y\\
			-x +3y & -x+y & -x-y
		\end{bmatrix}\\
	    &\ + \frac{\ell_1}{144}
		\begin{bmatrix}
			3 & 3 & 3\\
			-5 & -5 & -5\\
			-5 & -5 & -5
		\end{bmatrix}
		+ \frac{\ell_2}{144}
		\begin{bmatrix}
			3 & -5 & -5\\
			3 & -5 & -5\\
			3 & -5 & -5
		\end{bmatrix}.\\
	\end{aligned}
\end{equation*}
Here, $\Im_3$ is the $3\times 3$ identity matrix, and $\Jm_3$ is the $3\times 3$ matrix whose entries are all $1$.

The $36$ entries of $\Zm_D$ coincide with the formulas that Moriwaki found in \cite[Section 3]{Moriwaki2}. The code can be found in \texttt{Ex2.sage}.

\subsection{Epsilon invariant on a tesseract}

In this example, we compute an epsilon invariant on the metrized graph $\G$ formed by a tesseract using the two different expressions, namely Equations~\eqref{eq:EpsInvMor} and \eqref{eq:EpsInvCin}, and check if the results indeed coincide.

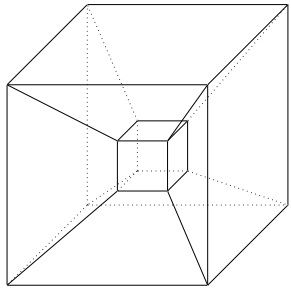
\begin{figure}[h]
	\centering
	
	\scalebox{0.7}{
	
	\begin{tikzpicture}[every node/.style={circle, draw, fill, minimum size=0pt, inner sep=0pt}]
		\pgfsetxvec{\pgfpoint{.3in}{0in}}
		\pgfsetyvec{\pgfpoint{0in}{.3in}}

		\foreach \point / \id in {
			(0,0)/0001,
			(0,5)/0011,
			(5,0)/1001,
			(5,5)/1011,
			(2,2)/0101,
			(2,7)/0111,
			(7,2)/1101,
			(7,7)/1111,
			(2.75,2.35)/0000,
			(2.75,3.60)/0010,
			(4.00,2.35)/1000,
			(4.00,3.60)/1010,
			(3.25,2.85)/0100,
			(3.25,4.10)/0110,
			(4.50,2.85)/1100,
			(4.50,4.10)/1110}
		{
			\node (\id) at \point {};
		}

		\path 
		(0011) edge (1011) edge (0111) edge (0001)
		(1001) edge (0001) edge (1101) edge (1011)
		(1111) edge (1101) edge (1011) edge (0111)
		(0010) edge (1010) edge (0110) edge (0000)
		(1000) edge (0000) edge (1100) edge (1010)
		(1110) edge (1100) edge (1010) edge (0110)
		(0000) edge (0001)
		(0010) edge (0011)
		(1000) edge (1001)
		(1010) edge (1011);

		\path[dotted]
		(0101) edge (1101) edge (0001) edge (0111)
		(0100) edge (1100) edge (0000) edge (0110)
		(0100) edge (0101)
		(1100) edge (1101)
		(0110) edge (0111)
		(1110) edge (1111);
	\end{tikzpicture}
	}
	
	\caption{A three-dimensional representation of a tesseract.}\label{fig:Tesseract}
\end{figure}

The vertices and edges of $\G$ are the usual vertices and edges of a tesseract, see \cref{fig:Tesseract}. More concretely,
\[V(\G) = \big\{(b_0,b_1,b_2,b_3) \mid b_i \in \{0,1\}\big\} \subset \mathbb{R}^4,\]
where we enumerate the sixteen vertices as 
\[(b_0,b_1,b_2,b_3) = p_{b_0 + 2b_1 + 4b_2 + 8b_3}.\]
The corresponding edge set is then given by
\[E(\G) = \big\{(p_i,p_j) \mid i<j \text{ and } \left|p_i-p_j\right| = 1\big\},\]
and each edge has length $1$. For the divisor $D = \sum_{k=0}^{15} kp_k$,
the two different expressions for $\eD$ yield, unsurprisingly, the same result:
\[\eD = \frac{7875}{122}.\]
See the file \texttt{Ex3.sage}.

\subsection{Epsilon invariant on a banana graph}
\label{subsec:ExBanana}

Consider the metrized graph $\G$ illustrated in \cref{fig:BananaGraph}, and define
\[\bq\colon \G\to \Z,\ \ \ s\mapsto 0.\]
The corresponding canonical divisor is
\[K = p_0 + p_1,\]
and therefore the pair $(\G,\bq)$ is a polarized metrized graph; see \cref{def:PolMetGr}. 

\begin{figure}[h]
	\centering
	\begin{tikzpicture}
		\node[vertex, label=left:$p_0$] (0) at (-2,0) {};
		\node[vertex, label=right:$p_1$] (1) at (2,0) {};

		\draw (0) to [out=80,in=100,looseness=1] (1) node[midway, above=1.2]{$a$};
		\draw (0) to (1) node[midway, above=0]{$b$};
		\draw (0) to [out=-80,in=-100,looseness=1] (1) node[midway, below=0.7]{$c$};
	\end{tikzpicture}
	\caption{A metrized ``banana'' graph.}\label{fig:BananaGraph}
\end{figure}
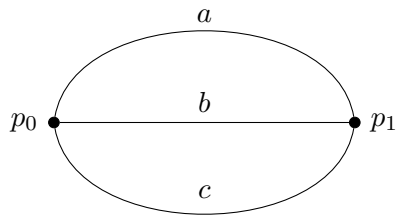

The epsilon invariant $\epsilon_K$ corresponding to the canonical divisor $K$ was computed by Moriwaki in \cite[Section~3, Type~VII]{Moriwaki3}, and his value is
\[\epsilon_K = \frac{2}{27}\prnths{a+b+c} + \frac{abc}{ab + ac + bc}.\]
Cinkir noticed that this value is wrong, and the correct value should be
\[\epsilon_K = \frac{1}{6}\prnths{a+b+c + \frac{abc}{ab + ac + bc}};\]
see \cite[Example~4.35 and Remark~4.36]{CinkirThesis}. 

In this final example, we compute the invariant $\epsilon_K$ using our techniques. The value we obtain is the same as Cinkir's value. The graph we use in our computation is represented in \cref{fig:BananaGraphAdeq}, and the code for this example can be found in \texttt{Ex4.sage}.

\begin{figure}[h]
	\centering
	\begin{tikzpicture}[decoration={markings,mark = at position 0.55 with {\arrow[thick]{latex}}}] 
		\node[vertex, label=left:$p_0$] (0) at (-2,0) {};
		\node[vertex, label=right:$p_1$] (1) at (2,0) {};

		\node[vertex, label=above:$p_2$] (2) at (0,1) {};
		\node[vertex, label=below:$p_3$] (3) at (0,-1) {};

		\draw [postaction=decorate] (0) -- (1) node[pos=0.5,above=0.1]{$e_0$};
		\draw [postaction=decorate] (0) -- (2) node[pos=0.45,above=0.1]{$e_1$};
		\draw [postaction=decorate] (2) -- (1) node[pos=0.55,above=0.1]{$e_2$};
		\draw [postaction=decorate] (0) -- (3) node[pos=0.45,below=0.1]{$e_3$};
		\draw [postaction=decorate] (3) -- (1) node[pos=0.55,below=0.1]{$e_4$};
		
		\qquad
		
		\filldraw (6,0.75) node{$L_1=L_2 = a/2$};
		\filldraw (6,0) node{$L_0=b$};
		\filldraw (6,-0.75) node{$L_3=L_4 = c/2$};
		
	\end{tikzpicture}
	\caption{The metrized graph in \cref{fig:BananaGraph} with an adequate vertex set.}\label{fig:BananaGraphAdeq}
\end{figure}

\addcontentsline{toc}{section}{References}
\printbibliography

\bigskip

\noindent{\bf Authors' addresses:}

\bigskip

\noindent Ruben Merlijn van Dijk, Bernoulli Institute, University of Groningen, Nijenborgh 9, 9747 AG, Groningen, The Netherlands.
\hfill \url{rvndijk@proton.me} 

\vspace{4mm}

\noindent Enis Kaya, Max Planck Institute for Mathematics in the Sciences, Inselstraße 22, 04103, Leipzig, Germany.
\hfill \url{enis.kaya@mis.mpg.de}

\end{document}